\documentclass[a4paper,10pt]{article}
\usepackage{amsmath}
\usepackage{amssymb}
\usepackage{amscd}
\usepackage{amsthm}

\usepackage{pdfsync}
\newtheorem{theorem}{Theorem}
\newtheorem{lemma}[theorem]{Lemma}

\newtheorem{remark}[theorem]{Remark}
\newtheorem{example}[theorem]{Example}
\newtheorem{result}[theorem]{Result}

\newtheorem{corollary}[theorem]{Corollary}

\def\PG{\mathrm{PG}}

\def\F{\mathbb{F}_q}
\def\S{\mathcal{S}}
\def\C{\mathcal{C}}
\def\Ha{\mathcal{H_\mathcal{S}}}
\title{Partial covers of $\PG(n,q)$}
\author{S. Dodunekov \and L. Storme \and G. Van de Voorde \thanks{This author is supported by the Fund for Scientific Research Ð Flanders (FWO - Vlaanderen).} }
\date{}
\begin{document}

\maketitle
\begin{abstract} In this paper, we show that a set of $q+a$ hyperplanes, $q>13$, $a\leq (q-10)/4$, that does not cover $\PG(n,q)$, does not cover at least $q^{n-1}-aq^{n-2}$ points, and show that this lower bound is sharp. If the number of non-covered points is at most $q^{n-1}$, then we show that all non-covered points are contained in one hyperplane. 
Finally, using a recent result of Blokhuis, Brouwer, and Sz\H{o}nyi \cite{BBS}, we remark that the bound on $a$ for which these results are valid can be improved to $a<(q-2)/3$ and that this upper bound on $a$ is sharp.
\end{abstract}

\section{Introduction}
Let $\PG(n,q)$ denote the $n$-dimensional projective space
over the finite field  $\F$ with $q$ elements, where $q=p^h$,  $p$ prime, $h\geq 1$. We denote the number of points in $\PG(n,q)$ by $\theta_n$, i.e., $\theta_n=\frac{q^{n+1}-1}{q-1}$.

Let $\C$ be a family of $q+a$ hyperplanes of $\PG(n,q)$. Denote by $\C(P)$ the set of hyperplanes of $\C$ containing $P$. A $(q+a)$-cover $\C$ of $\PG(n,q)$ is a family $\C$ of $q+a$ different hyperplanes in $\PG(n,q)$ such that $\vert\C(P)\vert\geq 1, \forall P\in\PG(n,q)$. A {\em partial} $(q+a)$-cover $\S$ is a set of $q+a$ hyperplanes such that there is at least one point $Q$ in $\PG(n,q)$ such that $\vert \mathcal{S}(Q)\vert =0$. A point $H$ for which $\vert\S(H)\vert=0$, is called a {\em hole of $\S$}. We denote the set of holes of $\S$ by $\Ha$.

A {\em blocking set} of $\PG(n,q)$ is a
set $B$ of points such that each hyperplane of $\PG(n,q)$ contains at least one point
of $B$. A blocking set $B$ is called {\em trivial}
 if it contains a line of $\PG(n,q)$. If a hyperplane contains
exactly one point of a blocking set $B$ in $\PG(n,q)$, it is called a {\em tangent hyperplane} to $B$, and a point $P$ of
$B$ is called {\em essential} when it belongs to a tangent hyperplane to $B$.
A blocking set $B$ is called {\em minimal} when no proper subset of $B$
is also a blocking set, i.e., when each point of $B$ is essential. 
\\

It is clear that a cover of $\PG(n,q)$ is a dual blocking set. Dualizing the above definitions yields that a cover $\C$ is called {\em trivial} if it contains all hyperplanes through a certain $(n-2)$-space and {\em minimal} if no proper subset of $\C$ is a cover. A hyperplane $\pi$ is {\em essential} to a cover $\C$ if there is a point $P\in\pi$ such that $\C(P)=\lbrace \pi\rbrace.$\\

The following reducibility results will be used throughout this article.
\begin{result} \cite[Remark 3.3]{Sz}\label{unique}  A blocking set of size at most $2q$ in $\PG(2,q)$ is uniquely reducible to a minimal blocking set. \end{result}
\begin{result}\cite[Corollary 1]{LSV2} \label{unique2} A blocking set of size smaller than $2q$ in $\PG(n,q)$ is uniquely reducible to a minimal blocking set. \end{result}

In Theorem \ref{hulp2}, we extend the following result of Blokhuis and Brouwer to general dimension.
\begin{result}\cite{BB} \label{hulp} Let $B$ be a blocking set in $\PG(2,q)$. If $\vert B \vert=2q-s$, then there are at least $s+1$ tangent lines through each essential point of $B$.
\end{result}
Finally, for $q$ a prime, we use the following result, proven by Blokhuis \cite{B} for $n=2$.
\begin{result}  \cite{bruen}\label{alg} \label{blokhuis} Let $B$ be a non-trivial blocking set in $\PG(n,p)$, where $p$ is an odd prime. Then
$$\vert B\vert \geq 3(p+1)/2.$$
\end{result}

\section{Partial covers of $\PG(2,q)$}\label{s1}

Throughout this section, $\S$ will denote a partial $(q+a)$-cover of $\PG(2,q)$, with $0\leq a\leq(q-10)/4$, $q>13$.

\begin{theorem}\label{st1}
If $\vert \Ha\vert \leq q+a$, then $\vert \Ha \vert \leq q$, and the holes are collinear.
\end{theorem}
\begin{proof} Let $\vert \Ha \vert=x$. Suppose that there are three non-collinear points in $\Ha$, otherwise the theorem is proven. The set $\Ha$ can be covered by at most $(x+1)/2$ lines, denote the set of these lines by $\mathcal{L}$.  Let $\mathcal{L}'$ be the minimal cover of $\Ha$ contained in $\mathcal{L}$. The set $\S\cup \mathcal{L'}$ is a cover in $\PG(2,q)$. Since the size of $\S\cup\mathcal{L'}$ is at most $q+a+(q+a+1)/2\leq 2q$, there is a unique minimal cover $\C$ contained in $\S\cup\mathcal{L'}$ (Result \ref{unique}). 

Let $\ell_y\in \mathcal{L'}$ be a $y$-secant to $\Ha$ with $y\leq (q-3a-1)/2$. Interchanging $\ell_y$ by $y$ other lines gives, together with the lines of $\S$, another cover $\C'$, with $\vert \C \cup \C'\vert\leq q+a+(q+a+1)/2+(q-3a-1)/2\leq 2q$. Hence, by the unique reducibility property, there is a unique minimal cover contained in $\C\cup \C'$, hence in $\C\cap \C'$. This minimal cover does not contain $\ell_y$, hence $\mathcal{L}'$ contains only lines with at least $(q-3a-1)/2$ holes.

If $\vert \mathcal{L'}\vert=1$, then the theorem is proven. Remark that if there is only one long secant and there are $q$ holes, then $a=0$.

 Suppose that $\vert\mathcal{L'}\vert=z$. These $z$ secants, together with the $q+a$ lines of $\S$, form a cover $\C''$. Then there is a line $L$ in $\mathcal{L}'$ with less than $(q+a+1+ {z \choose 2})/z$ holes. Suppose to the contrary that any line in $\mathcal{L}'$ contains at least $(q+a+1+ {z \choose 2})/z$ holes, then there are at least $z(q+a+1+ {z \choose 2})/z-{z\choose 2}=q+a+1>q+a$ holes, a contradiction. 
We construct a new cover by replacing this line $L$ with less than $(q+a+1+ {z \choose 2})/z$ lines, one through each hole on $L$. In total, with the $z$ secants and the lines of $\S$, this set of lines constitutes a cover $\C'''$ of size at most $q+a+z+(q+a+1+ {z \choose 2})/z$. If 
\begin{eqnarray}
q+a+z+(q+a+1+ {z \choose 2})/z\leq 2q,\label{vgl}
\end{eqnarray}
 the unique reducibility property (Result \ref{unique}) shows that there is a minimal cover contained in $\C''\cap \C'''$, which does not contain the line $L$. This implies that the line $L$ was not essential to the cover $\C''$, a contradiction. It is easy to check that for $z\geq 2$ and $z<9$, inequality (\ref{vgl})  holds for $a\leq(q-10)/4$ and $q>13$. Hence, there are at least $9$ long secants essential to the minimal cover $\C''$. On each of these secants, there are at least $(q-3a-1)/2$ holes, hence we have at least $9(q-3a-1)/2-9\cdot8/2$ holes. But
 $$9(q-3a-1)/2-36>q+a$$ if $a<(7q-63)/25$. Since $a\leq (q-10)/4$, and $(q-10)/4<(7q-63)/25$, the theorem follows.\end{proof}

 \begin{corollary}\label{priemg} Let $q$ be a prime. If $\vert \Ha\vert \leq  q+a$, then $\S$ consists of $q$ lines through the same point $R$ and $a$ lines $l_1,\ldots, l_a$, not through $R$.
 \end{corollary}
 \begin{proof}
 It follows from Theorem \ref{st1} that the holes are contained in one line, say $M$. Then the lines of $\S$, together with $M$, constitute a cover $\C$ of size $q+a+1<3(q+1)/2$. Result \ref{blokhuis}, together with Result \ref{unique}, shows that the unique minimal cover contained in $\C$ is the set of all lines through a point $R$. It is clear that the line $M$ is one of the lines through $R$. The other $a$ lines are random, but do not contain $R$.
 \end{proof}

\section{Partial covers of $\PG(n,q)$}

Before extending the results of Section \ref{s1} to general dimension, we need the extension of Result \ref{hulp}.

\begin{theorem}\label{hulp2} The number of tangent hyperplanes through an essential point of a blocking set $B$ of size $q+a+1$, $\vert B\vert \leq 2q$, in $\PG(n,q)$ is at least $q^{n-1}-aq^{n-2}$.
\end{theorem}

\begin{proof} The arguments of this proof are based on the proof of Proposition 2.5 in \cite{SzW}.

 For $n=2$, Result \ref{hulp} proves this theorem. Assume by induction that the theorem holds for all dimensions $i\leq n-1$. Let $B$ be a blocking set in $\pi=\PG(n,q)$. Since $\vert B \vert\leq 2q$, there is an $(n-2)$-space $L$ in $\pi$ that is skew to $B$. Let $H$ be a hyperplane through $L$. Embed $\pi$ in $\PG(2n-2,q)$.  Let $P$ be a $\PG(n-3,q)$, skew to $\pi$, in $\PG(2n-2,q)$. Then $\langle B,P\rangle$, the cone with vertex $P$ and base $B$, is a blocking set with respect to the $(n-1)$-spaces of $\PG(2n-2,q)$. Let $H^{\ast}\neq H$ be a hyperplane through $L$ only sharing one point $Q$ with $B$. Since $\vert B \vert$ is at most $2q$, there are at least $2$ tangent hyperplanes through $L$, hence $H^{\ast}$ can be chosen different from $H$.

 Let $\mathcal{S}$ be a regular $(n-2)$-spread through $L$ and $\langle Q,P\rangle$ in $W$, the $(2n-3)$-dimensional space spanned by $L$ and $\langle Q,P\rangle$. Using the Andr\'e-Bruck-Bose construction (see \cite{Br}), this yields a projective plane $\PG(2,q^{n-1})=\Pi^{W}$. The arguments of \cite[Proposition 2.5]{SzW} show that $H$ defines a line $\ell$ in $\Pi^{W}$, only having essential points of the blocking set $\bar{B}$ of size $1+(q+a)q^{n-2}=q^{n-1}+aq^{n-2}+1$, where $\bar{B}$ is the blocking set in $\PG(2,q^{n-1})$, corresponding to $\langle B,P\rangle$. This number of points comes from $\langle Q,P\rangle$ at infinity, which is one point of the blocking set, and the $q+a$ affine points $R_i$ of $B$, all on a cone $\langle R_i,P\rangle$ with $q^{n-2}$ affine points.
Result \ref{hulp} shows that any essential point  of $\bar{B}$ lies on at least $q^{n-1}-aq^{n-2}$ tangent lines to the blocking set $\bar{B}$ in $\Pi^{W}$. We will show that the number of tangent lines through an essential point of the blocking set $\bar{B}$ in $\Pi^{W}$ is a lower bound on the number of tangent hyperplanes through an essential point of $B$ in $\PG(n,q)$.

A tangent line through an affine essential point $R$ of $\bar{B}$ corresponds to an $(n-1)$-space $\langle R,\Omega\rangle$, with $\Omega$ a spread element of $\mathcal{S}$. The space $\langle R,\Omega\rangle$ is not necessarily a tangent hyperplane to $B$ in $\PG(n,q)$. Note that $\Omega\neq \langle Q,P\rangle$, since both are spread elements and cannot coincide since $\langle Q,P\rangle$ is an element of the blocking set, hence $\langle R,Q,P\rangle$ cannot be a tangent space.

The projection of $\langle R,\Omega\rangle$ from $P$ onto $\PG(n,q)$ is an $(n-1)$-dimensional space through $R$ in $\PG(n,q)$ which is skew to $Q$ since $\Omega\cap \langle Q,P \rangle =\emptyset$, and which only has $R$ in common with $B$ since $\langle\Omega,R\rangle\cap\langle B,P\rangle=\lbrace R \rbrace$. Hence, this projection is a tangent $(n-1)$-space through $R$ to $B$ in $\PG(n,q)$. So we have shown that any tangent line in $R$ to $\bar{B}$ gives rise to a tangent hyperplane to $B$ in $R$. If any tangent line to $\bar{B}$ in $R$ gives rise to a different tangent hyperplane to $B$, the theorem is proven.

Let $\eta$ be a tangent hyperplane to $B$ in $R$ which is the projection of two tangent lines $\langle \Omega,R\rangle$ and $\langle \Omega', R\rangle$. The dimension of $\langle \eta,P\rangle$ is $2n-3$, and $\dim(\langle \eta,P\rangle\cap W)=2n-4$.
A hyperplane of $\PG(2n-3,q)$ contains exactly one element of a regular $(n-2)$-spread. Since it contains $\Omega$ and $\Omega'$, $\Omega=\Omega'$. So $\eta$ is the projection of at most one such $(n-1)$-space.

For every essential point $Q$ of $B$, it is possible to select a tangent hyperplane $H$ through $Q$, and to let this tangent hyperplane $H$ play the role described in the preceding paragraph. Since $Q$ is an affine essential point, this implies that $Q$ lies in at least $q^{n-1}-aq^{n-2}$ tangent hyperplanes to $B$.
\end{proof}

\begin{lemma} \label{gev}
Let $\S$ be a partial $(q+a)$-cover of $\PG(n,q)$, $a<q$. If all holes of $\S$ are contained in a hyperplane $\pi$ of $\PG(n,q)$, then $|\Ha|\geq q^{n-1}-aq^{n-2}$.
\end{lemma}
\begin{proof} The hyperplanes of $\S$, together with the hyperplane $\pi$ that contains all holes, form a cover of size $q+a+1$, in which $\pi$ is an essential hyperplane. Dualizing gives a blocking set $B$ of size $q+a+1$, where the dual of $\pi$ is an essential point. Theorem \ref{hulp2} shows that the dual of $\pi$ lies on at least $q^{n-1}-aq^{n-2}$ tangent hyperplanes to $B$. Dualizing again shows that $\pi$ contains at least $q^{n-1}-aq^{n-2}$ points that are only covered by $\pi$. Removing $\pi$ shows that there are at least $q^{n-1}-aq^{n-2}$ holes.\end{proof}

\begin{remark} The lower bound in Lemma \ref{gev} is sharp. Let $\S$ be the set of $q$ hyperplanes through a fixed $(n-2)$-space $\pi_{n-2}$. Let $H$ be the hyperplane through $\pi_{n-2}$, which is not chosen. Take $\mathrm{a}$ hyperplanes for which the $(n-2)$-dimensional intersections with $H$ are all distinct and go through a common $(n-3)$-space of $\pi_{n-2}$, then there are exactly $q^{n-1}-aq^{n-2}$ holes.
\end{remark}
From now on, $\S$ denotes a partial $(q+a)$-cover of $\PG(n,q)$, $n\geq 3$.
We denote the following property by $(A_x)$:
\begin{itemize}
\item[$(A_x)$] If $\mathcal{S}$ is a partial $(q+b)$-cover in $\PG(2,q)$, $b\leq x<(q-2)/3,$ with at most $q+b$ holes, then $q-b\leq |\Ha|\leq q$ and the holes are collinear.
\end{itemize}
Note that we have shown  in Theorem \ref{st1} and Lemma \ref{gev} that the property $(A_x)$ holds for $x\leq(q-10)/4$, $q>13$.
\begin{lemma} \label{nieuw} Assume $(A_x)$ for all $x\leq a$. If a partial cover $\mathcal{S'}$ of $\PG(2,q)$ contains 3 non-collinear holes, then $\vert \S'\vert > q+a$.
\end{lemma}
\begin{proof} If $|\S'|=q+a$, this follows immediately from property $(A_x)$. So suppose that $\vert \mathcal{S}'\vert=q+x'$, $x'<a$, and that there are 3 non-collinear holes, say $H_1,H_2,H_3$. Let $P$ be a point not on $H_1H_2$, $H_1H_3$, or $H_2H_3$. Adding $a-x'$ lines through $P$, different from $PH_1$, $PH_2$, $PH_3$, to the partial cover $\S'$ gives a cover $\S''$, with $|\S''|=q+a$. Applying property $(A_x)$ to $\S''$, the corollary follows.
\end{proof}

\begin{lemma} \label{lem2} Assume $(A_x)$ for all $x\leq a$, and $\vert \mathcal{H}_\S\vert\leq q^{n-1}$. 
A line that contains $2$ holes of $\S$, contains at least $a+3$ holes of $\S$.
\end{lemma}
\begin{proof} Let $L$ be a line with $t$ holes, $t<q-a$, and let $\pi$ be a plane through $L$. Assumption $(A_x)$ shows that if $\pi$ contains at most $q+a$ holes, there are at least $q-a$ holes, which are all collinear, a contradiction. Hence, every plane through $L$ contains at least $q+a+1$ holes, which implies that there are at least 
$$\theta_{n-2}(q+a+1-t)+t$$ holes in $\PG(n,q)$, which has to be at most $q^{n-1}$. If $t=a+2$, $\theta_{n-2}(q+a+1-a-2)+a+2>q^{n-1}$,  a contradiction. Hence, $t$ is at least $a+3$.
\end{proof}

\begin{lemma}\label{lem1} Assume $(A_x)$ for all $x\leq a$, and $\vert \mathcal{H}_\S\vert\leq q^{n-1}$. Every hole of $\S$ lies on more than $q^{n-2}/2$ lines with at least $q-a$ holes.
\end{lemma}
\begin{proof} Let $R$ be a hole. There is a line $L$ through $R$ containing only covered points and $R$, otherwise there would be at least $\theta_{n-1}+1$ holes. Using assumption $(A_x)$ and Lemma \ref{nieuw}, we see that a plane through $L$ contains either at most $q-1$ holes on a line through $R$, different from $L$, or it contains at least $q+x$ holes different from $R$.

Suppose that there are $X$ planes through $L$ with at most $q-1$ holes different from $R$. Using assumption $(A_x)$ and Lemma \ref{nieuw}, we see that the number of holes is at least
$$X(q-a-1)+(\theta_{n-2}-X)(q+a)+1,$$
which has to be at most $q^{n-1}$. Putting $X=q^{n-2}/2$ yields a contradiction. Hence, there are more than $q^{n-2}/2$ planes with at most $q$ holes. Again using assumption $(A_x)$, we see that in each of these planes, there is a line through $R$ containing at least $q-a-1$ other holes, and all holes in such a plane lie on this line.
\end{proof}

\begin{theorem}\label{st3} Assume $(A_x)$ for all $x\leq a$, and $\vert \mathcal{H}_\S\vert\leq q^{n-1}$. Then the holes of $\S$ are contained in one hyperplane of $\PG(n,q)$.
\end{theorem}
\begin{proof} For $n=2$, this is assumption $(A_x)$ with $x=a$. Suppose by induction that this theorem holds for any dimension $i\leq n-1$.

First, we show that there is a hyperplane $\pi$ of $\PG(n,q)$ with at most $q^{n-2}$ holes. Let $R$ be a hole. There is a line $L$ through $R$ containing only covered points and $R$. Suppose that all planes through $L$ contain more than $q$ holes, then there would be at least $\theta_{n-2}q+1$ holes, a contradiction. Suppose that there is a $d$-dimensional space $\pi_d$ with at most $q^{d-1}$ holes. Then there is a $(d+1)$-dimensional space containing $\pi_d$ with at most $q^{d}$ holes. Otherwise, the number of holes would be at least $\theta_{n-d-1}(q^d+1-q^{d-1})+q^{d-1}$, a contradiction if $d \leq n-1$. Hence, by induction, there is a hyperplane $\pi$ of $\PG(n,q)$ with at most $q^{n-2}$ holes.

Using the induction hypothesis, all holes in $\pi$ are contained in an $(n-2)$-dimensional space $\pi_{n-2}$ of $\pi$. Moreover, Lemma \ref{gev} shows that the number of holes in $\pi_{n-2}$ is at least $q^{n-2}-aq^{n-3}$.

There are at least $\theta_{n-2}(q-a-1)+1$ holes in $\PG(n,q)$ since every plane through $L$ contains at least $q-a-1$ extra holes. Hence, there is certainly a hole $R'$ that is not contained in $\pi_{n-2}$.

Now we distinguish between two cases.

{\bf Case 1:} All lines through $R'$ with at least $q-a$ holes are lines which intersect $\pi_{n-2}$. Lemma \ref{lem1} shows that there are at least $q^{n-2}/2$ such lines. Since a line through two holes contains at least $a+3$ holes (see Lemma \ref{lem2}), counting the holes in $\langle R',\pi_{n-2}\rangle$ yields that this number  is at least 
$$q^{n-2}(q-a-1)/2+(q^{n-2}-aq^{n-3}-q^{n-2}/2)(a+2)+1.$$

If all holes are contained in  $\langle R',\pi_{n-2}\rangle$, the theorem is proven. Suppose now that not all holes are contained in the hyperplane $\langle R',\pi_{n-2}\rangle$. Let $R''$ be a hole not in $\langle R',\pi_{n-2}\rangle$. Connecting $R''$ with all the holes in  $\langle R',\pi_{n-2}\rangle$ yields at least  $(a+2)(q^{n-2}(q-a-1)/2+(q^{n-2}-aq^{n-3}-q^{n-2}/2)(a+2)+1)+1$ holes, which is more than $q^{n-1}$, a contradiction.

{\bf Case 2: } There is a line through $R'$, skew to $\pi_{n-2}$, with at least $q-a$ holes. This yields at least 
$$(q-a)(q^{n-2}-aq^{n-3})(a+1)+q^{n-2}-aq^{n-3}+q-a>q^{n-1}$$ holes, a contradiction.
\end{proof}

\begin{theorem} \label{algemeen} Assume $(A_x)$ for all $x\leq a$, then the number of holes of $\S$ is at least $q^{n-1}-aq^{n-2}$.
\end{theorem}
\begin{proof} This follows immediately from Theorem \ref{st3} and Lemma \ref{gev}.
\end{proof}

\begin{corollary}\label{priemg2} Assume $(A_x)$ for all $x\leq a$, and $\vert \mathcal{H}_\S\vert\leq q^{n-1}$. If $q$ is a prime, $\S$ consists of $q$ hyperplanes through a common $(n-2)$-space $\pi$ and $a$ other hyperplanes, not through $\pi$.
 \end{corollary}
 \begin{proof}
 It follows from Theorem \ref{st3} that the holes are contained in one hyperplane, say $\mu$. Then the hyperplanes of $\S$, together with $\mu$, constitute a cover $\C$ of size $q+a+1<3(q+1)/2$. Result \ref{blokhuis}, together with Result \ref{unique2}, shows that the unique minimal cover contained in $\C$ is the set of all hyperplanes through an $(n-2)$-space $\pi$. Since this set covers $\PG(n,q)$ entirely, the hyperplane $\mu$ is one of the hyperplanes through $\pi$. The other $a$ hyperplanes are random, but do not contain $\pi$.
 \end{proof}

As remarked before, assumption $(A_x)$ holds for all partial $(q+a)$-covers of $\PG(n,q)$, $x\leq (q-10)/4$, $q>13$. 

However, a recent result of Blokhuis, Brouwer and Sz\H{o}nyi \cite{BBS} shows that assumption $(A_x)$ is valid for all $x$, where $x<(q-2)/3$. Moreover, the following example shows that the upper bound $a<(q-2)/3$ is sharp.

\begin{example} Let  $a=(q-2)/3$ and let $\S$ be a set of $q-1$ lines $L_i$ through a point $P$, and $a+1$ other lines through a fixed point, lying on one of the lines $L_i$. Then there are $2(q-a-1)=q+a$ holes, lying on two lines.
\end{example}

Combining Theorems \ref{st3}, \ref{algemeen} and Corollary \ref{priemg2} with the result of Brouwer, Blokhuis and Sz\H{o}nyi, yields the following theorem.
\begin{theorem} If $\S$ is a partial $(q+a)$-cover of $\PG(n,q)$, $a<(q-2)/3$, with at most $q^{n-1}$ holes, then there are at least $q^{n-1}-aq^{n-2}$ holes and the holes are contained in one hyperplane. If $q$ is a prime, $\S$ consists of $q$ hyperplanes through a common $(n-2)$-space $\pi$ and $a$ other hyperplanes, not through $\pi$.

\end{theorem}

{\bf Acknowledgements:} This research was supported by the Project {\em Combined algorithmic and theoretical study of combinatorial structures} between the Fund for Scientific Research Flanders-Belgium (FWO-Flanders) and the Bulgarian Academy of Sciences. The third author wants to thank the people at the Bulgarian Academy of Sciences in Veliko Tarnovo and Sofia for their hospitality during her research visit in the framework of this project.

The authors thank the referee and A.E. Brouwer for their valuable comments.

Address of the authors: \\
\vspace{5pt}
\\
Stefan Dodunekov:\\
IMI / MFI\\
8, Acad. G. Bonchev\\
1113 Sofia (Bulgaria) \\
stedo@moi.math.bas.bg\\
\\
 Leo Storme, Geertrui Van de Voorde:\\
Department of pure mathematics and computer algebra\\
Ghent University\\
Krijgslaan 281-S22\\
9000 Ghent  (Belgium)\\
$\lbrace$ls,gvdvoorde$\rbrace$@cage.ugent.be\\
http://cage.ugent.be/ $\sim$ $\lbrace$ls,gvdvoorde$\rbrace$\\
\\


\begin{thebibliography}{99}
\bibitem{B}A. Blokhuis, On the size of a blocking set in ${\rm PG}(2,p)$.  {\em Combinatorica } {\bf 14}  (1994),  no. 1, 111--114.
\bibitem{BB}A. Blokhuis and A.E. Brouwer, Blocking sets in Desarguesian projective
planes. {\em Bull. London Math. Soc.} {\bf 18} (1986), 132--134.
\bibitem{BBS} A. Blokhuis, A.E. Brouwer, and T. Sz\H{o}nyi, Covering all points except one. {\em Preprint}.
\bibitem{Br} R.H. Bruck and R.C. Bose, The construction of translation planes from projective spaces.  {\em J. Algebra}  {\bf 1}  (1964), 85--102. 
\bibitem{bruen}A.A. Bruen, Blocking sets and skew subspaces of projective space. {\em Canad. J. Math.} {\bf 32} (1980), 628--630.

\bibitem{LSV2}M. Lavrauw, L. Storme, and G. Van de Voorde, On the code generated by the incidence matrix of points
and $k$-spaces in $\PG(n,q)$ and its dual.  {\em Finite Fields Appl.} {\bf 116} (2009), 996--1001.

\bibitem{Sz}T. Sz\H{o}nyi,  Blocking sets in Desarguesian affine and projective planes. {\em  Finite Fields Appl.} {\bf  3} (1997), 187--202.
\bibitem{SzW} T. Sz\H{o}nyi and Zs. Weiner, Small blocking sets in higher dimensions. {\em J. Combin. Theory, Ser. A} {\bf 95} (2001), 88--101.


\end{thebibliography}
\end{document}